\documentclass[final]{siamltex}

\usepackage{amsfonts}
\usepackage{amssymb}
\usepackage{amsmath}
\usepackage{latexsym}
\usepackage{graphicx}
\usepackage{geometry}
\usepackage{hyperref}
\usepackage{color}

\usepackage{multicol}
\usepackage[bottom]{footmisc}
\usepackage{url}

\usepackage{bm}

\newcommand{\ds}{\displaystyle}

\newcommand{\sfrac}[2]{\mathchoice
  {\kern0em\raise.5ex\hbox{\the\scriptfont0 #1}\kern-.15em/
   \kern-.15em\lower.25ex\hbox{\the\scriptfont0 #2}}
  {\kern0em\raise.5ex\hbox{\the\scriptfont0 #1}\kern-.15em/
   \kern-.15em\lower.25ex\hbox{\the\scriptfont0 #2}}
  {\kern0em\raise.5ex\hbox{\the\scriptscriptfont0 #1}\kern-.2em/
   \kern-.15em\lower.25ex\hbox{\the\scriptscriptfont0 #2}}
  {#1\!/#2}}

\def\ub {{\bf u}}

\def\ub {{\bf u}}
\def\xb {{\bf x}}

\def\fb {{\bf f}}

\def\myhalf {\sfrac{1}{2}}
\def\threehalf {\sfrac{3}{2}}

\begin{document}

\title{Analysis of some projection method based preconditioners for models of incompressible flow}
\author{Mingchao Cai \thanks{Courant Institute of Mathematical Sciences, New York University, New York, NY 10012, E-mail: cmchao2005@gmail.com}
}

\date{\today}          
\maketitle

\begin{abstract}
In this paper, several projection method based preconditioners for various incompressible flow models are studied. In particular, we are interested in the theoretical analysis of a pressure-correction projection method based preconditioner \cite{griffith2009accurate}. For both the steady and unsteady Stokes problems, we will show that the preconditioned systems are well conditioned. Moreover, when the flow model degenerates to the mixed form of an elliptic operator, the preconditioned system is an identity no matter what type of boundary conditions are imposed; when the flow model degenerates to the steady Stokes problem, the multiplicities of the non-unitary eigenvalues of the preconditioned system are derived. These results demonstrate the effects of boundary treatments and are related to the stability of the staggered grid discretization. To further investigate the effectiveness of these projection method based preconditioners, numerical experiments are given to compare their performances. Generalizations of these preconditioners to other saddle point problems will also be discussed.
\end{abstract}

\begin{keywords}
projection method, preconditioning, Stokes equations, saddle point problems
\end{keywords}

\pagestyle{myheadings}
\thispagestyle{plain}
\markboth{ANALYSIS OF PROJECTION METHOD BASED PRECONDITIONERS}{MINGCHAO CAI}

\section{Introduction}
We consider to solve the incompressible Stokes equations
\begin{equation}\label{Time_Stokes}
 \left \{
        \begin{array} {ll}
         {\rho}\frac{\partial{\bf u}}{\partial t} -\mu {\bf \Delta} {\bf u}+\nabla p
                     =\mathbf{f}       \quad &\mbox{in}~\Omega,\\
                 -\nabla\cdot{\bf u}=g \quad &\mbox{in}~\Omega,
        \end{array}
               \right.
\end{equation}
subject to suitable boundary conditions on $\partial \Omega$. Here, ${\bf u}(\xb, t)$ is the velocity, $p(\xb, t)$ is the pressure, ${\rho}(\xb, t)$ is the density function, $\mu(\xb, t)$ is the viscosity, ${\bf f}$ may include an external force and the nonlinear term $(\ub \cdot \nabla) \ub$, $g$ is 0 or a source term. For simplicity, we assume that $g=0$, $\rho$ and $\mu$ are constants.


Different methods have been proposed and studied for the solution of the above problem. The original projection method, introduced by Chorin and Temam \cite{chorin1968numerical, temam1968method}, decouples the computation of velocity and pressure into solving several Poisson problems. It yields time stepping schemes for the incompressible Stokes equations which require relatively simple linear solvers. However, this approach complicates the specification of physical boundary conditions and may cause artificial numerical boundary layers \cite{e1995projection}. Appropriate "artificial" boundary conditions must be prescribed for the intermediate variables ${\bf u}^*$ and $\phi$, and obtaining high-order accuracy for ${\bf u}$ and $p$ may not be possible in some cases \cite{e1995projection, guermond2006overview}. Compared with the original projection method, implicit schemes which couple the computations of $\ub$ and $p$ together can provide accurate numerical solution and possess better stability \cite{guermond2006overview, newren2007unconditionally}. However, implicit coupled schemes require efficient solvers for the resulted saddle point systems, which are usually difficult to solve. Using preconditioned Krylov methods provides a good remedy to overcome the difficulties \cite{benzi2005numerical, ipsen2001note, murphy2000note}. Consequently, designing effective preconditioner is of vital importance. In a recent work \cite{griffith2009accurate}, a pressure-correction projection method \cite{perot1993analysis, griffith2009accurate, quarteroni2000factorization} was adopted as the preconditioner for the resulting saddle point system. Numerical experiments \cite{griffith2009accurate, cai2013efficient} show that the preconditioner is effective and efficient. However, we note that the preconditioner in \cite{griffith2009accurate} is restricted to the time dependent Stokes case and the corresponding theoretical justification is absent. In this work, we are interested in the following aspects: generalizing the preconditioner in \cite{griffith2009accurate} to steady Stokes problem; providing a theoretical analysis for both the steady and unsteady problems; studying some related preconditioners which are based on exact or spectral-equivalent approximations of the Schur complement; building up the connections between our work with some existing works; extending the projection method based preconditioners to other saddle point problems.

In this paper, we use an implicit time scheme and a staggered grid spatial discretization for (\ref{Time_Stokes}). The resulting saddle point system y using a preconditioned GMRES method \cite{saad1986gmres, saad2003iterative}. One of our interests is analyzing the pressure-correction projection method based preconditioner \cite{griffith2009accurate}. The obtained results are as follows. For both the steady and unsteady Stokes problems, it is shown that the preconditioned systems are well conditioned. More precisely, the nontrivial eigenvalues of the preconditioned systems have uniform lower and upper bounds which are independent of the mesh refinement and the physical parameters. Moreover, for steady Stokes equations, the multiplicities of the non-unitary eigenvalues of the preconditioned system are derived based on the analysis of the commutator difference operator. Specifically, for two dimensional Dirichlet boundary value problem, if there are $n$ cells along each direction, it is shown that there are at most $4(n-1)$ eigenvalues not equal to 1. For periodic boundary value problem, the preconditioned operator is an identity operator. Meanwhile, motivated by other related works \cite{benzi2005numerical, benzi2006augmented, elman1999preconditioning, ipsen2001note, kobelkov2000effective, mardal2004uniform, mardal2011preconditioning, murphy2000note, olshanskii2006uniform, olshanskii2012multigrid}, we propose and investigate several projection method based preconditioners. These preconditioners use exact or approximate inverse of the Schur complement. Their effectiveness, advantages and disadvantages, and the connections with those existing works will be highlighted. In addition, we will discuss how to generalize our work to other saddle point problems.


This paper is organized as follows. In section 2, we present the time and the spatial discretization of the Stokes problem, the projection method based preconditioners, the corresponding matrix representations, and other related preconditioners. In section 3, eigenvalue analysis of the preconditioned systems are presented. In section 4, numerical experiments are given to compare the performances of different preconditioners. Concluding remarks are drawn in the last section.

\section{Discretizations and precondtioners}

\subsection{Time and spatial discretizations}
To have better stability and accuracy properties, an implicit time stepping scheme is adopted. In this paper, we apply the following backward Euler scheme.
\begin{equation}\label{Euler_scheme}
 \left \{
        \begin{array} {ll}
        \frac{\rho}{\Delta t}(\ub^{k+1}-\ub^{k}) - \mu {\bf \Delta} \ub^{k+1} + \nabla p^{k+1}  = \fb^{k+1}, \\
-\nabla\cdot \ub^{k+1} =0.
  \end{array}
               \right.
\end{equation}
Here, $\ub^{k}$ and $p^{k}$ are the approximate solution of $\ub$ and $p$ at $t=k \Delta t$.

\begin{figure}\label{Fig1}
\centering
\includegraphics[height=5.8cm]{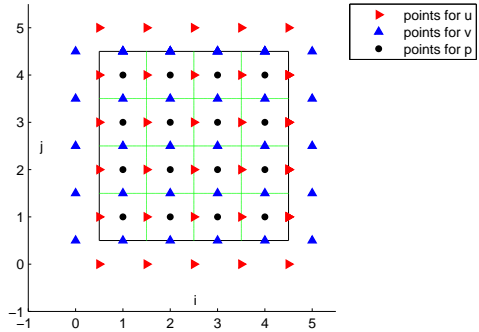}
\vskip .1cm
\caption{A staggered-grid spatial discretization and the spatial location of the face-centered velocity and the
cell-centered pressure degrees of freedom.}
\end{figure}
The staggered grid (i.e. marker-and-cell or MAC \cite{harlow1965numerical}) method is applied for the spatial discretization. In this work, we assume that the computational domain $\Omega$ is a two dimensional rectangle and there are $n_x$ and $n_y$ cells along the $x-$ direction and the $y-$ direction respectively. For simplicity, let us further assume that $\overline{\Omega}=[0,1]^2$ and $h_x=h_y=h$ (and therefore $n_x=n_y=n$). To present the staggered grid discretization, we introduce the following sets.
$$
\displaystyle \overline{\Omega}_1 = \left\{\left(\left(i+\frac{1}{2}\right)h, jh \right) : 0 \le i \le n, 0 \le j \le n+1  \right\},
$$
$$
\displaystyle \overline{\Omega}_2 = \left\{\left(ih, \left(j+\frac{1}{2}\right)h \right) : 0 \le i \le n+1, 0 \le j \le n \right\},
$$
$$
\displaystyle {\Omega}_3 = \left\{\left(ih, jh \right) : 1 \le i \le n, 1 \le j \le n  \right\}.
$$

As shown in Figure 2.1, $\overline{\Omega}_1$, $\overline{\Omega}_2$ and $\overline{\Omega}_3$ are the point sets for $u$, $v$ and $p$ respectively. The divergence of ${\bf u}=(u, v)^{T}$ is approximated at cell centers by
$D {\bf u} =D^x u +D^y v$ with
$$
(D^x u)_{i, j} =
\frac{u_{i+\myhalf, j}-u_{i-\myhalf, j}}{h}   \quad \quad
(D^y v)_{i, j} =
\frac{v_{i, j+\myhalf}-v_{i, j-\myhalf}}{h}.
$$
$D {\bf u}$ is well defined at all points in $\Omega_3$. The gradient of $p$ is approximated at the $x-$ and $y-$ edges of the grid cells by
$Gp =(G^x p, G^y p)^{T}$ with
$$
(G^x p)_{i-\myhalf, j}
=\frac{p_{i,j}- p_{i-1,j}}{h} \quad \quad (G^y p)_{i, j-\myhalf}=
\frac{p_{i, j}-p_{i, j-1}}{h}.
$$
We see that $G^x$ is well defined in the interior of $\Omega_1$ and $G^xy$ is well defined in the interior of $\Omega_2$. There are three different approximations to the Laplace operator required in the staggered grid scheme, one defined at the cell centers,
one defined at the $x$-edges of the grid, and one defined at the $y$-edges of the grid. All employ the standard 5-point
finite difference stencil. The Laplacian of $p$ is approximated at cell centers by
$$
(L^c p)_{i,j}=\frac{p_{i+1, j}- 2 p_{i, j}+p_{i-1, j}}{h^2}+\frac{p_{i, j+1}- 2 p_{i, j}+p_{i, j-1}}{h^2}.
$$
The approximations to the Laplacians of $u$ and $v$ are
$$
(L^x u)_{i-\myhalf,j}=\frac{u_{i+\myhalf, j}- 2 u_{i-\myhalf, j}+u_{i-\threehalf, j}}{h^2}+\frac{u_{i-\myhalf, j+1}- 2 u_{i-\myhalf, j}+u_{i-\myhalf, j-1}}{h^2},
$$
$$
(L^y v)_{i,j-\myhalf}=\frac{v_{i+1, j-\myhalf}- 2 v_{i, j-\myhalf}+v_{i-1, j-\myhalf}}{h^2}+\frac{v_{i, j+\myhalf}- 2 v_{i, j-\myhalf}+v_{i, j-\threehalf}}{h^2}.
$$
The finite difference approximation to the vector Laplacian of ${\bf u}$ is denoted as $L{\bf u}=(L^x u, L^y v)^T$. We see that $L^x u$ is well defined in $\Omega_1$ and $L^y v$ is well defined in $\Omega_2$.

To evaluate the above finite difference approximations near the boundaries of $\Omega$, we need the specification of boundary values along $\partial \Omega$ and "ghost values" located outside of $\Omega$. For periodic boundary conditions, the ghost values at boundaries are copies of the corresponding interior values. For example, the ghost values at the west boundary are copied from the corresponding function values which are evaluated at the nodes nearest to the east boundary. To demonstrate the treatment of Dirichlet boundary condition, we consider the discretization of $u$ near the south boundary. By using the standard stencils, we evaluate
\begin{equation}\label{u_yy_BD_old}
\left[\frac{\partial}{\partial y}\left(\mu \frac{\partial u}{\partial y}\right)\right]_{i+\myhalf,1} = \mu \left(\frac{u_{i+\myhalf, 2} - u_{i+\myhalf, 1}}{h^2}\right) - \mu \left(\frac{2u_{i+\myhalf,1}-2u_{i+\myhalf, \myhalf}}{h^2}  \right).
\end{equation}
Here, $u_{i+\myhalf, \myhalf}$ is the $x$- component of the velocity at the node $((i+\frac{1}{2})h, \frac{1}{2}h)$. Similar treatments are applied to $\frac{\partial}{\partial x}\left(\mu \frac{\partial v}{\partial x}\right)$ near the west and east boundaries. This type of boundary treatment for keeping the symmetry of the discrete system can be interpreted as follows. If ${\bf u}_i$ is defined on $\Omega_i$ (i=1,2), then we denote the function which is defined on $\overline{\Omega}_i$ and coincides with ${\bf u}_i$ on $\Omega_i$ as $\tilde{{\bf u}}_i$. If ${\bf x}\in \partial \Omega \cap \Omega_i$, we then set $\tilde{{\bf u}}_i=0$. If ${\bf x}\in \partial \Omega_i / \partial \Omega$, we set $\tilde{{\bf u}}_i({\bf x})=-{\bf u}_i({\bf x}')$ with ${\bf x}'$ being the nearest node of $\partial \Omega_i$ to ${\bf x}$. By this way, the 0 Dirichlet boundary condition is weakly imposed and for the points in $\Omega_1 \times \Omega_2$, and there holds $(L^x u)_{i-\myhalf,j}=(L^x \tilde{u})_{i-\myhalf,j}, (L^y v)_{i,j-\myhalf}=(L^y \tilde{v})_{i,j-\myhalf}$. For Dirichlet boundary condition, the discretization based on (\ref{u_yy_BD_old}) for approximating the second order derivatives has local truncation order $\mathcal{O}(1)$ while the overall discretization is spatially globally second-order accurate \cite{hundsdorfer2003numerical, newren2007unconditionally}.

After the time discretization and the staggered grid spatial discretization, the fully discrete system is of the following saddle point form.
\begin{equation}\label{Saddle_operator}
M \left[\begin{array}{c}
{\bf u} \\
p
\end{array}\right]=
\left[\begin{array}{cc}
A  & G\\
-D & 0
\end{array}\right]
\left[\begin{array}{c}
{\bf u} \\
p
\end{array}\right]=
\left[\begin{array}{c}
{\bf f}_{\ub} \\
g_p
\end{array}\right].
\end{equation}
Here, $A=\frac{\rho}{\Delta t}I - \mu L$ is the discrete operator for $\frac{\rho}{\Delta t} I- \mu{\bf \Delta}$, $D$ for $\nabla \cdot$, and $G$ for $\nabla$, ${\bf f}_{\ub}$ includes both the force terms and the velocity terms from the previous time step, and $g_p$ represents the source term and the boundary data. Moreover, no matter what type of boundary conditions, we note that the staggered grid discretization always satisfies the compatibility: $-D=G^*$. That is, the two operators are adjoint to each other. Moreover, $DG=L^c$ with $L^c$ being the discrete Laplacian operator applied to the pressure variable. As previously introduced, $L^c$ is well-defined at cell centers. In addition, the boundary condition for $L^c$ is naturally determined by the boundary condition of the global system: when the boundary conditions for the Stokes system are Dirichlet type, $L^c$ is imposed with Neumann boundary conditions, when boundary conditions for the Stokes system are periodic, $L^c$ is also imposed with periodic boundary conditions.

\subsection{Projection method based preconditioners}

The pressure-correction projection algorithm \cite{kim1985application, griffith2009accurate, perot1993analysis, quarteroni2000factorization} for the constant density and constant viscosity Stokes equations is as follows.
\begin{itemize}
\item {\it Step 1:} Solve for an intermediate velocity ${\bf u}^*$ using an implicit viscosity step:
\begin{equation}\label{implicit_velocity}
A {\bf u}^*= {\bf f}_{\ub}.
\end{equation}

\item {\it Step 2:} Note that $\ub^*$ does not generally satisfy the discrete incompressibility constraint, we project ${\bf u}^*$ to the divergence free space by solving
\begin{equation}\label{time_update}
 \left \{
        \begin{array} {ll}
\rho\frac{{\bf u}^{k+1}-{\bf u}^*}{\Delta t}= - G\phi,
\\
- D{\bf u}^{k+1} = 0.
  \end{array}
   \right.
\end{equation}
This leads to a Poisson equation for the intermediate variable $\phi$:
\begin{equation}\label{pre_Poisson_eqn}
-\left(DG\right) \phi =  -  \frac{\rho}{\Delta t}\left(D{\bf u}^* \right).
\end{equation}

\item {\it Step 3:} Update the velocity and correct the pressure term by
\begin{equation}\label{vel_correction}
{\bf u}^{k+1}={\bf u}^* - \frac{\Delta t}{\rho} G\phi.
\end{equation}
\begin{equation}\label{pressure_correction}
p^{k+1}=\phi -  \frac{\Delta t}{\rho} \mu\left(DG\right)\phi.
\end{equation}
\end{itemize}

This algorithm is analogous to the second order projection method whic is based on the Crank-Nicolson scheme \cite{kim1985application, griffith2009accurate}. From (\ref{implicit_velocity})-(\ref{pressure_correction}), the above projection algorithm can be written into a factorization form.
\begin{equation}\label{Boyce_precon_equ}
P_1^{-1}=\left[\begin{array}{cc}
I  & - G\\
0  & \frac{\rho}{\Delta t}I- \mu L^c
\end{array}\right]
\left[\begin{array}{cc}
I     &  0 \\
0     &  (-L^c)^{-1}
\end{array}\right]
\left[\begin{array}{cc}
I     &  0 \\
-D&  -I
\end{array}\right]
\left[\begin{array}{cc}
A^{-1}       & 0 \\
0            & I
\end{array}\right].
\end{equation}
Here $-L^c=G^* G$ is rank one deficient. $(-L^c)^{-1}p$ is in the sense that $G p \notin \mbox{Ker}(G^*)$.
As pointed out in \cite{cai2013efficient}, the last 2 steps can be combined together and therefore leads to the following
more efficient implementation.
\begin{equation}\label{better_P1}
P_1^{-1}=\left[\begin{array}{cc}
I     & G({L^c})^{-1} \\
0     & \mathbb{S}^{-1}
\end{array}\right]
\left[\begin{array}{cc}
I     &   0 \\
-D& -I
\end{array}\right]
\left[\begin{array}{cc}
{A}^{-1}   & 0 \\
0      & I
\end{array}\right].
\end{equation}
Here,
\begin{equation}\label{app_Schur_complement}
\mathbb{S}^{-1}= \frac{\rho}{\Delta t}(-L^c)^{-1} + {\mu} I 
\end{equation}
is an approximation of the inverse of the Schur complement
\begin{equation}\label{Schur_complement}
\displaystyle S^{-1}= \left(-DA^{-1} G \right)^{-1}.
\end{equation}
The approximation property of (\ref{app_Schur_complement}) to (\ref{Schur_complement}) was firstly found in \cite{cahouet1988some} and has been justified using the Fourier analysis approach \cite{cahouet1988some, cai2013efficient, kobelkov2000effective} and the operator mapping theory \cite{benzi2006augmented, mardal2004uniform, mardal2011preconditioning, olshanskii2006uniform, olshanskii2012multigrid, turek1999efficient}.

Note that the pressure correction (\ref{pressure_correction}) is based on the following arguments. Multiplying by $\left( \frac{\rho}{\Delta t} I-\mu L \right)$ on both sides of (\ref{time_update}), and noting that ${\bf u}^{k+1}$ and $p^{k+1}$ satisfy the discrete Stokes equation (\ref{Euler_scheme}) and ${\bf u}^*$ satisfies (\ref{implicit_velocity}), we derive that
\begin{equation}\label{pre_corr_der}
-Gp^{k+1} =- \frac{\Delta t}{\rho} \left( \frac{\rho}{\Delta t} I-\mu L \right) G\phi.
\end{equation}
Multiplying by $D$ on both sides of (\ref{pre_corr_der}), the least square solution of (\ref{pre_corr_der}) satisfies
$$
{L^c} p^{k+1}={L^c} \phi  -\mu\frac{\Delta t}{\rho} DL G\phi.
$$
Therefore, we see that the pressure correction can be
\begin{equation}\label{pre_corr2}
p^{k+1}= \phi - \frac{\Delta t}{\rho}  \mu(L^c)^{-1}DLG\phi.
\end{equation}
By assuming that there holds the commutation property $LG=G{L^c}$, we obtain (\ref{pressure_correction}). Generally, the commutation property does not hold and $LG=G{L^c}$ only in special cases, such as the boundary conditions are periodic.

From the above discussions, if the pressure correction step is modified to be (\ref{pre_corr2}), then we obtain
\begin{equation}\label{modified_precon_equ}
P_4^{-1}=\left[\begin{array}{cc}
I & - G\\
0      & \frac{\rho}{\Delta t}I+\mu(-L^c)^{-1} DL G
\end{array}\right]
\left[\begin{array}{cc}
I     &  0 \\
0     &  (-L^c)^{-1}
\end{array}\right]
\left[\begin{array}{cc}
I     &  0 \\
-D&  -I
\end{array}\right]
\left[\begin{array}{cc}
A^{-1}       & 0 \\
0            & I
\end{array}\right].
\end{equation}
Furthermore, for steady Stokes equation, $A=-\mu L$. By setting $\rho=0$ in (\ref{Boyce_precon_equ}), one obtains the projection method based preconditioner for the steady Stokes problem:
\begin{equation}\label{P1_steady}
\displaystyle
P_1^{-1}=\left[\begin{array}{cc}
I & - G\\
0      & - \mu L^c
\end{array}\right]
\left[\begin{array}{cc}
I     &  0 \\
0     &  (-L^c)^{-1}
\end{array}\right]
\left[\begin{array}{cc}
I     &  0 \\
-D&  -I
\end{array}\right]
\left[\begin{array}{cc}
A^{-1}    & 0 \\
0                & I
\end{array}\right].
\end{equation}
From our study \cite{cai2013efficient}, the projection algorithm expressed as (\ref{P1_steady}) can not be directly applied as a solver for the stationary Stokes. For instance, if (\ref{P1_steady}) is used for solving the driven cavity problem, one can not resolve the corner vortices. However, we note that (\ref{P1_steady}) is still a robust preconditioner for (\ref{Saddle_operator}).

Motivated by the theory developed in \cite{murphy2000note, ipsen2001note} and the above discussion, we see that one can use the exact inverse of the Schur complement. Then, we obtain the following projection type preconditioner.
\begin{equation}\label{P1_ext}
\widetilde{P}^{-1}_1=\left[\begin{array}{cc}
I     &  G(D G)^{-1} \\
0     &  S^{-1}
\end{array}\right]
\left[\begin{array}{cc}
I     &   0 \\
-D    & -I
\end{array}\right]
\left[\begin{array}{cc}
{A}^{-1}   & 0 \\
0      & I
\end{array}\right]
\end{equation}

{\bf Remark 1.} In this work, the derivations of the preconditioners are based on the backward Euler scheme (\ref{Euler_scheme}). The derivation based on the Crank-Nicolson scheme can be found in \cite{griffith2009accurate}. In addition, we note that $P_4$ is similar to the BFBt preconditioner in \cite{elman1999preconditioning} if the Oseen equations degenerates to the Stokes equations.

\subsection{Other related preconditioners}
As ${\mathbb{S}}^{-1}$ is a good approximation of the inverse of the Schur complement, we consider to extend the block diagonal preconditioners discussed in \cite{cahouet1988some, mardal2004uniform, mardal2011preconditioning,  olshanskii2006uniform, olshanskii2012multigrid} to the following lower and upper triangular preconditioners. We study
\begin{equation}\label{CC_Pre_lower}
P_2^{-1}=\left[\begin{array}{cc}
I       & 0 \\
0       & -{\mathbb{S}}^{-1}
\end{array}\right]
\left[\begin{array}{cc}
I     &  0 \\
D     &  I
\end{array}\right]
\left[\begin{array}{cc}
{A}^{-1} & 0 \\
0      & I
\end{array}\right]
\end{equation}
and
\begin{equation}\label{CC_Pre_upper}
P_3^{-1}=\left[\begin{array}{cc}
{A}^{-1}   & 0 \\
0    & I
\end{array}\right]
\left[\begin{array}{cc}
I     & -G \\
0 &  I
\end{array}\right]
\left[\begin{array}{cc}
I       & 0 \\
0   & -{\mathbb{S}}^{-1}
\end{array}\right].
\end{equation}
These preconditioners are actually the approximations of the following block triangular preconditioners
\begin{equation}\label{saddle_Schur}
\widetilde{P}_2^{-1} = \left[\begin{array}{cc}
A       & 0 \\
-D& -S
\end{array}\right]^{-1}  \quad \mbox{and} \quad
\widetilde{P}_3^{-1} = \left[\begin{array}{cc}
A     & G \\
0 & -S
\end{array}\right]^{-1},
\end{equation}
respectively.

In the following, we will study all the preconditioners mentioned above. In particular, we will discuss the advantage, the disadvantages, and the algebraic properties of all the precondtioners. 

\subsection{Matrix representations}
In this subsection, we present the matrix representations of the discrete differential operators. If the lexicographic order \cite{elman2005finite} (indexing the unknowns from the left to the right and from the bottom to the top) of the degrees of freedom is used, from the staggered grid discretization and the boundary treatments, then the resulting linear system for the 2D Stokes problem with Dirichlet boundary on a square domain is a saddle point problem of the form (\ref{Saddle_operator}). The corresponding matrices are \cite{elman1999preconditioning}
$$
-L=\frac{1}{h^2}\left[
\begin{array}{cc}
I_n\otimes T_D + T_E \otimes I_{n-1}    &0   \\
 0       & I_{n-1} \otimes T_E + T_D \otimes I_n
\end{array}
\right],    \quad
G^*=\frac{1}{h}\left[
\begin{array}{cc}
I_n\otimes B_D     &B_D \otimes I_n
\end{array}
\right].
$$
Here, $L$ and $G^*$ are the matrix representations of the discrete ${\bf \Delta}$ and the discrete $-\mbox{div}$ according to the above variables ordering, and
$$
T_D=\left[
\begin{array}{ccccc}
2         & -1      &      &      &       \\
-1        & 2       & -1        &     &   \\
    & \ddots      & \ddots    & \ddots   & \\
    &   & -1     & 2        & -1\\
    &      &     & -1       & 2
\end{array}
\right], \quad
T_E=\left[
\begin{array}{ccccc}
3         & -1      &      &         & \\
-1        & 2       & -1        &    & \\
    & \ddots      & \ddots    & \ddots     & \\
    &   & -1     & 2        & -1\\
    &   &      & -1       & 3
\end{array}
\right],
$$

$$
T_N=\left[
\begin{array}{ccccc}
1         & -1      &      &      &    \\
-1        & 2       & -1        &     & \\
    & \ddots     & \ddots    & \ddots & \\
    &   & -1     & 2        & -1\\
    &   &        & -1       & 1
\end{array}
\right], \quad
E_0=\left[
\begin{array}{ccccc}
1       &       &      &       \\
        & 0       &    &    \\
    &      & \ddots    &    \\
    &   &    & 0       \\
    &   &     &  &   1
\end{array}
\right], \quad
B_D=\left[
\begin{array}{cccc}
-1        &      &      &        \\
1        & -1     &      &       \\
    & \ddots     & \ddots  &     \\
    &          & 1       & -1 \\
    &     &         &  1
\end{array}
\right].
$$
Moreover, there holds \cite{elman1999preconditioning}
$$
B_D B_D^{T}= T_N, \quad B_D^{T}B_D= T_D, \quad T_E=T_N+2E_0, \quad B_DT_D B_D^{T}= T_N^2.
$$

If periodic boundary conditions are imposed in both the $x$- and $y$- directions, then
$$
-L=\frac{1}{h^2}\left[
\begin{array}{cc}
I_n\otimes T_P + T_P \otimes I_{n}    &0   \\
 0       & I_{n} \otimes T_P + T_P \otimes I_n
\end{array}
\right],    \quad
G^*=\frac{1}{h}\left[
\begin{array}{cc}
I_n\otimes B_P     &B_P \otimes I_n
\end{array}
\right].
$$
Here $T_P$ and $B_P$ are circulant matrices \cite{davis1979circulant, strang1999discrete}.
$$
T_P=\left[
\begin{array}{ccccc}
2         & -1      &      &      &    -1   \\
-1        & 2       & -1        &     &   \\
    & \ddots      & \ddots    & \ddots   & \\
    &   & -1     & 2        & -1\\
-1    &      &     & -1       & 2
\end{array}
\right], \quad
B_P=\left[
\begin{array}{ccccc}
1       &-1      &         &       &   \\
        & 1      &-1       &       &   \\
        &      & \ddots  &\ddots   &   \\
        &      &           & 1     & -1\\
-1      &       &          &       & 1
\end{array}
\right].
$$
Moreover, we have
$$
B_P B_P^{T}=T_P, \quad T_P=F^* \Lambda F.
$$
Here, $F$ is a Fourier matrix and $\Lambda$ is a diagonal matrix \cite{strang1999discrete} with the eigenvalues as its diagonal entries. Therefore, when boundary conditions are purely periodic, there holds the commutation property: $LG=G{L^c}$. Moreover, a Fourier matrix times a vector can be realized by applying the Fast Fourier transform to the vector. Therefore, FFT is usually used for the resulted system \cite{strang1999discrete}.

Based on the above discussions and explanations, if the boundary condition is periodic in the $x$- direction and is Dirichlet in the $y$- direction, it is not difficult to write out the corresponding matrix representations either.

\section{Analysis}

To link $P_1$ with the original saddle point form (\ref{Saddle_operator}), we formally calculate
\begin{eqnarray}\label{preconditioned_sys}
P_1^{-1} M &=&
\left[\begin{array}{cc}
I & - G\\
0      & \frac{\rho}{\Delta t}I- \mu L^c
\end{array}\right]
\left[\begin{array}{cc}
I     &  0 \\
0     &  (-L^c)^{-1}
\end{array}\right]
\left[\begin{array}{cc}
I      &   0 \\
-D&  -I
\end{array}\right]
\left[\begin{array}{cc}
A^{-1}       & 0 \\
0            & I
\end{array}\right]
\left[\begin{array}{cc}
A       & G\\
-D& 0
\end{array}\right]  \cr
&=& \left[\begin{array}{cc}
I        & (I-G(L^c)^{-1}D)A^{-1} G\\
0        & \mathbb{S}^{-1} S
\end{array}\right].
\end{eqnarray}
Here, we assume that $L^c$ is invertible. In implementation, $(-L^c)^{-1}$ times a pressure vector is evaluated in a way which ensures that the null-space is handled consistently and the mean pressure value is kept constant. Note that $A = \frac{\rho}{\Delta t} I - \mu L$, when the boundary condition is purely periodic, the basic operators $G$, $D$, $A$ and $L^c$ are commutative, $P_1^{-1} M$ is exactly the discrete identity operator. Most importantly, from (\ref{preconditioned_sys}), we see that the preconditioned system is block upper triangular. Therefore,
\begin{equation}\label{eigen_eqn}
\displaystyle \mbox{det}\left(\lambda I - P_1^{-1} M\right) = \mbox{det}\left((\lambda-1) I\right) \mbox{det}\left(\lambda I -\mathbb{S}^{-1} S \right).
\end{equation}
It means that the eigenvalues of the preconditioned system are either 1 or the eigenvalues of the $(2,2)$-block of (\ref{preconditioned_sys}).

Similarly, for the other preconditioners, it is easy to derive that
\begin{eqnarray}\label{preconditioned_P2M}
P_2^{-1} M &=&
\left[\begin{array}{cc}
I        & A^{-1} G \\
0    & \mathbb{S}^{-1} S
\end{array}\right],
\end{eqnarray}
and (or calculate $M P_3^{-1}$)
\begin{eqnarray}\label{preconditioned_P3M}
P_3^{-1} M &=&
\left[\begin{array}{cc}
I-A^{-1} G \mathbb{S}^{-1} D        & A^{-1} G \\
\mathbb{S}^{-1} D   &  0
\end{array}\right].
\end{eqnarray}
From (\ref{preconditioned_sys})-(\ref{preconditioned_P3M}), we see that $P_1^{-1} M$, $P_2^{-1} M$ and $P_3^{-1} M$ actually share the same spectral. Moreover, the eigenvalues of the preconditioned systems are either 1 or the eigenvalues of the $\mathbb{S}^{-1} S$. Therefore, we shall analyze the spectral of $\mathbb{S}^{-1} S$ under the staggered grid discretization.
It should be pointed out that, in (\ref{preconditioned_sys}), the $(1, 2)$ block contains a projection operator $I-G(L^c)^{-1}D$. The projection operator, maps a velocity vector to $\mbox{Ker}(D)$. Intuitively, (\ref{preconditioned_sys}) has better approximation properties than (\ref{preconditioned_P2M}) or (\ref{preconditioned_P3M}). We will check that whether $P_1$ does provide better performance or not.

For the exact inverse Schur complement based precondtioner, $\widetilde{P}^{-1}_1$, we have the following proposition, which is an extension of Remark 2 in \cite{murphy2000note}.

\begin{proposition} Let $\widetilde{P}^{-1}_1$ be the preconditioner defined in (\ref{P1_ext}), we have
\begin{equation}\label{preconded_sys_ext}
\widetilde{P}^{-1}_1 M =
\left[\begin{array}{cc}
I        & (I-G(L^c)^{-1}D)A^{-1} G\\
0        & I
\end{array}\right].
\end{equation}
The preconditioned system $\widetilde{P}^{-1}_1 M$ have the minimal polynomial $(\lambda-1)^2$.
\end{proposition}

\subsection{Analysis of the preconditioners for the stationary Stokes problem}
By setting $\rho=0$ in (\ref{Boyce_precon_equ}), (\ref{modified_precon_equ}) and (\ref{P1_ext}), we obtain the projection method based precondtioner (\ref{P1_steady}) for the stationary Stokes problem. From (\ref{eigen_eqn}) and noting that $\mathbb{S}^{-1}= \mu I$ in the stationary Stokes case, we see that the eigenvalues of the preconditioned system are either 1 or the eigenvalues of the Schur complement (\ref{Schur_complement}). In this subsection, we mainly discuss the preconditioned system which uses (\ref{P1_steady}) as the preconditioner. In particular, we will focus on the analysis for the Dirichlet boundary problem.

Let ${\bf U}$ be the linear space of vector functions defined on $\Omega_1 \times \Omega_2$ and ${\bf P}$ be the linear space defined on $\Omega_3$ satisfying mean value 0. For homogeneous Dirichlet boundary problem, the boundary treatments are stated in section 2.1. Moreover, we endow ${\bf U}$ and ${\bf P}$ with the following inner products and norms.
$$
({\bf u}, {\bf v})_{\bf U}=\displaystyle \sum_{{\bf x}\in \Omega_1} h^2 u_1({\bf x})v_1({\bf x})+\sum_{{\bf x}\in \Omega_2} h^2 u_2({\bf x}) v_2({\bf x}), \quad ||{\bf u}||_0^2= ({\bf u}, {\bf u})_{\bf U},
$$
$$
(p, q)_0=\displaystyle \sum_{{\bf x}\in \Omega_3} h^2 p({\bf x}) q({\bf x}), \quad ||p||_0^2=(p, p)_0,
$$
$$
||{\bf u}||_1^2=(-L{\bf u}, {\bf u})_{\bf U}=\displaystyle \sum_{{\bf x}\in \Omega_1} h^2 (-L^x u_1, u_1)+\sum_{{\bf x}\in \Omega_2} h^2(-L^y u_2, u_2).
$$
In the definition of $||\cdot||_1$ norm, we have implicitly used the 0 Dirichlet boundary conditions. We see that the functional spaces ${\bf U}$ and ${\bf P}$ are approximations of $(H^1_0(\Omega))^2$ and $L^2_0(\Omega)$ respectively \cite{kobelkov1986inequality, kobelkov1994numerical}.

The staggered grid discretization is uniformly div-stable \cite{girault1996finite, han1998new, kobelkov1986inequality, ol2000best}, it means that the following lemma holds \cite{kobelkov1986inequality}.

\begin{lemma}
\label{stability}
There exists a constant $\beta > 0$ independent of the mesh refinement such that
\begin{equation}\label{inf_sup}
\sup_{{\bf u}|\partial \Omega=0} \frac{(-D{\bf u}, p)_0}{||{\bf u}||_1}  \ge \beta ||p||_0, \quad \forall p \in {\bf P}.
\end{equation}
\label{inf_sup_stability}
\end{lemma}

The proof of this lemma in finite element terminology can be found in \cite{girault1996finite, han1998new}.
The constant $\beta$ depends on the shape of the domain (for example, the ratio of the length along the $x$- direction and the length along the $y$- direction \cite{ol2000best}) but independent of the mesh refinement. This lemma states that $D$ is a surjective from ${\bf U}$ to ${\bf P}$ and the staggered grid discretization is uniformly div-stable \cite{brezzi1991mixed}.

To study the spectral of the Schur complement, we introduce the following lemma. It points out that the spectrum of the Schur complement is equivalent to the solution of a generalized eigenvalue problem.

\begin{lemma}
\label{equivalent_eig}
If $\lambda$ and $p$ are an nonzero eigenvalue and the corresponding eigenvector of the Schur complement, i.e.,
\begin{equation}\label{eigen_Schur}
G^* A^{-1}Gp = \lambda p,
\end{equation}
then $\frac{1}{\lambda}$ is the eigenvalue of the generalized eigenvalue problem:
\begin{equation}\label{equ_eigen_Schur}
A {\bf u} = \frac{1}{\lambda} GG^* {\bf u}   \quad \mbox{with} \quad {\bf u}= Gp.
\end{equation}
Furthermore, if $\lambda$ is nonzero, the inverse of the above argument also holds.
\end{lemma}
\begin{proof}
Multiplying both sides of (\ref{eigen_Schur}) by $G$ and denoting ${\bf u}=Gp$, we see that
$$
GG^* A^{-1}{\bf u}  = \lambda {\bf u}.
$$
Moreover, as $A^{-1}(G G^*)$ and $(GG^*) A^{-1}$ are similar to each other, they share the same spectral. Therefore, it is sufficient to investigate the eigenvalue problem
$$
A^{-1} (GG^*){\bf u}  = \lambda {\bf u}.
$$
This eigenvalue problem is equivalent to the generalized eigenvalue problem (\ref{equ_eigen_Schur}) if $\lambda$ is nonzero. We therefore get the conclusion.  \qquad\end{proof}

From Lemma \ref{equivalent_eig}, we see that the nonzero eigenvalues of the Schur complement is equivalent to the generalized eigenvalue of (\ref{equ_eigen_Schur}). Moreover, because both $A^{-1}$ and $GG^* $ are symmetric, by using (\ref{equ_eigen_Schur}), we have
$$
\displaystyle
\frac{1}{\lambda_{max}} = \mbox{min}_{{\bf u} \not\in \mbox{Ker}(G^*)} \frac{(A {\bf u}, {\bf u})}{(GG^*{\bf u}, {\bf u})},  \quad
\mbox{and}
\quad
\frac{1}{\lambda_{min}} = \mbox{max}_{{\bf u} \not\in \mbox{Ker}(G^*)} \frac{(A {\bf u}, {\bf u})}{(GG^*{\bf u}, {\bf u})}.
$$
Here, $\lambda_{max}$ and $\lambda_{min}$ are the maximum and the minimum eigenvalues of the Schur complement $-DA^{-1}G$ respectively. Thus, to estimate the bounds of the spectral of the Schur complement, it is sufficient to study the generalized Rayleigh quotient,
\begin{equation}\label{Rayleigh_quo}
\frac{(A {\bf u}, {\bf u})}{(GG^*{\bf u}, {\bf u})}  \quad \forall {\bf u} \not\in \mbox{Ker}(G^*).
\end{equation}

{\bf Remark 2}. When (\ref{Time_Stokes}) degenerates to the stationary Stokes problem, one can absorb the inverse of the viscosity into the pressure term or one can scale ${\mu}$ in approximating the inverse of the Schur complement. Therefore, without loss of the generality, we will assume that $\mu=1$ and $A=-L$ in this section.

The following theorem states the eigenvalue analysis of the preconditioned system for the steady Stokes problem.

\begin{theorem}
\label{th:eig_steady}
For the steady Stokes problem with purely Dirichlet boundary condition, we have
(i) $\lambda (G^* A^{-1} G) \in \{0\} \cup [\beta^2, 1]$. The multiplicity of 0 eigenvalue is 1.
(ii). There are at most $4(n-1)$ eigenvalues not equal to 1. (If $n_x \neq n_y$, there are at most $2(n_x-1)+2(n_y-1)$ eigenvalues not equal to 1.) For boundary condition with the $x-$ direction periodic and the $y-$ direction Dirichlet, there are at most $2n_x$ eigenvalues not equal to 1.
\end{theorem}

\begin{proof}
(i) The multiplicity of the 0 eigenvalue is easy to understand because the dimension of the null space of the Schur complement $G^*A^{-1}G$ is 1. The lower bound of the eigenvalues of the Schur complement is $\beta^2$, which is independent of the mesh refinement \cite{elman2005finite, kobelkov1986inequality, ol2000best}. This result is derived from the definition of the discrete inf-sup condition \cite{elman2005finite, kobelkov1986inequality, ol2000best}. To see this, we check
\begin{eqnarray}\label{inf_eig_bound}
\displaystyle
\beta^2 \le \inf_{q \notin \mbox{Ker}(G)}\sup_{\bf u \in {\bf U}} \frac{(q, -D{\bf u})^2_0}{||{\bf u}||_1^2||q||_0^2} &=& \inf_{q \notin \mbox{Ker}(G)}\sup_{{\bf u}\in {\bf U}}\frac{(Gq, {\bf u})_{\bf U}^2}{(-L {\bf u}, {\bf u})_{\bf U}||q||_0^2} \cr
&=& \inf_{q \notin \mbox{Ker}(G)} \frac{1}{||q||_0^2} \sup_{{\bf v}=-L^{\myhalf}{\bf u}}\frac{(q, -G^*L^{-\myhalf}{\bf v})_{0}^2}{({\bf v}, {\bf v})_{\bf U}}\cr
&=& \inf_{q \notin \mbox{Ker}(G)} \frac{1}{||q||_0^2} \sup_{{\bf v}\in {\bf U}}\frac{(-L^{-\myhalf}Gq, {\bf v})_{\bf U}^2}{({\bf v}, {\bf v})_{\bf U}}\cr
&=& \inf_{q \notin \mbox{Ker}(G)} \frac{|(-L^{-\myhalf}Gq, -L^{-\myhalf}Gq)_{\bf U}|}{||q||_0^2} \cr
&=& \inf_{q \notin \mbox{Ker}(G)} \frac{(G^* A^{-1} Gq, q)_0}{||q||_0^2}.
\end{eqnarray}
Here $-L^{\myhalf}$ means the positive square root of $-L$ (that is $(-L^{\myhalf})^2=- L$).

To estimate the upper bound of the eigenvalues, we study the properties of the Rayleigh quotient (\ref{Rayleigh_quo}). Firstly, we note that
$$
\frac{(A {\bf u}, {\bf u})}{(GG^*{\bf u}, {\bf u})}= 1+\frac{((A-GG^* ){\bf u}, {\bf u})}{(GG^* {\bf u}, {\bf u})}, \quad \forall {\bf u} \notin \mbox{Ker}(G^*)
$$
To check the eigenvalues of the Schur complement are not greater than 1, it is sufficient to prove
\begin{equation}\label{Rayleigh_ineq}
\frac{((A-GG^* ){\bf u}, {\bf u})}{(GG^* {\bf u}, {\bf u})} \ge 0  \quad \mbox{for any} ~ {\bf u} \notin \mbox{Ker}(G^*).
\end{equation}
By using the matrix representations in the previous section and the properties of the Kronecker product, we have
\begin{equation}\label{curl_rot_Mat}
A-GG^* =\frac{1}{h^2}\left[
\begin{array}{cc}
 T_E \otimes I_{n-1}    & -B_D\otimes B_D^{T}   \\
 -B_D^{T} \otimes B_D       & I_{n-1} \otimes T_E
\end{array}
\right].
\end{equation}

For any $2\times 2$ block symmetric operator, there holds
$$
\left[
\begin{array}{cc}
H    & B^T   \\
B    & C
\end{array}
\right]= \left[
\begin{array}{cc}
I       & 0  \\
BH^{-1} & I
\end{array}
\right]
\left[
\begin{array}{cc}
H    & 0  \\
0    & C-BH^{-1}B^T
\end{array}
\right]
\left[
\begin{array}{cc}
I     & H^{-1}B^T  \\
0     & I
\end{array}
\right].
$$
Thus, to prove $A-GG^* $ is symmetric semi-positive definite, it is sufficient to prove $T_E \otimes I_{n-1}$ (SPD) and $I_{n-1} \otimes T_E - (B_D^{T} \otimes B_D) (T_E \otimes I_{n-1})^{-1} (B_D\otimes B_D^{T})$
are symmetric positive semi-definite. By using the properties of the Kronecker product,
$$
\begin{array}{c}
I_{n-1} \otimes T_E - (B_D^{T} \otimes B_D) (T_E^{-1} \otimes I_{n-1}) (B_D\otimes B_D^{T})= I_{n-1} \otimes T_E -(B_D^{T}T_E^{-1}B_D) \otimes T_N \\
=(I_{n-1}-(B_D^{T}T_E^{-1}B_D)) \otimes T_N + 2I_{n-1} \otimes E_0.
\end{array}
$$
Since $T_E-B_D^{T}B_D=E_0$ is positive semi-definite, we see that $(I_{n-1}-(B_D^{T}T_E^{-1}B_D))$ is positive semi-definite. Therefore (\ref{Rayleigh_ineq}) holds.

(ii). To show that there are at most $4(n-1)$ nonzero eigenvalues, we study the commutation difference operator \cite{guermond2006overview}.
$$
A G  - G L^c= A G- G(G^* G) =(A-GG^* )G.
$$
From (\ref{curl_rot_Mat}), we see that
\begin{eqnarray}\label{diff_curlrot}
(A-GG^* )G&=& \frac{1}{h^3}\left[
\begin{array}{cc}
 T_E \otimes I_{n-1}    & -B_D\otimes B_D^{T}   \\
 -B_D^{T} \otimes B_D       & I_{n-1} \otimes T_E
\end{array}
\right]
\left[
\begin{array}{c}
 I_n \otimes B_D^T      \\
 B_D^{T} \otimes I_{n-1}
\end{array}
\right] = \frac{2}{h^3}\left[
\begin{array}{c}
 E_0 \otimes B_D^T      \\
 B_D^{T} \otimes E_0
\end{array}
\right].
\end{eqnarray}
The right hand side is a matrix with its rank equals to $4(n-1)$. It means that $(A-GG^* )G$ are all zeros in the interior of the domain.  Moreover, for any nontrivial eigenvector ${\bf u} \notin \mbox{Ker}(G^*)$, there exists $p \in {\bf P}$ such that ${\bf u}=Gp$. Therefore, from Lemma \ref{equivalent_eig} and (\ref{diff_curlrot}), we see that there are at most $4(n-1)$ eigenvalues are not equal to 1.

If the boundary conditions are purely periodic, then all the nonzero eigenvalues of $S$ are equal to 1. From the commutation property and Lemma 3.2, the nontrivial eigenvalues of $S$ are equivalent to the generalized eigenvalues of $\lambda AGp = GG^*Gp$. As $AG=-GL^c$, all the nontrivial eigenvalues of $S$ are equal to 1. For the case with the boundary condition which is periodic in $x$- direction and Dirichlet in $y$- direction, the proof is similar to the above discussion. \qquad\end{proof}

{\bf Remark 3.} From Theorem \ref{th:eig_steady} and the inf-sup condition, we conclude that
\begin{equation}\label{Schur_bound}
\displaystyle
1 \le \frac{(-L \ub, \ub)}{(G G^* \ub, \ub)} \le \frac{1}{\beta^2},  \quad \forall \ub \notin \mbox{Ker}(G^*).
\end{equation}
The ideas of the proof of Theorem \ref{th:eig_steady} actually originate from some differential operator identities. Note that
in continuous level, if the function is smooth enough, the following identities of differential operators \cite{arnold2006differential} hold.
\begin{equation}\label{Lap_iden}
-{\bf \Delta} = -\mbox{{\bf grad}}~\mbox{div} + \mbox{{\bf curl}}~\mbox{rot},
\end{equation}
and
\begin{equation}\label{diff_iden2}
-{\bf \Delta}~\mbox{{\bf grad}} = -\mbox{{\bf grad}}~\mbox{div}~\mbox{{\bf grad}},
\end{equation}
as
\begin{equation}\label{divcurl_iden}
\mbox{rot}~\mbox{{\bf grad}}=0, \quad \mbox{div}~ \mbox{{\bf curl}}=0.
\end{equation}
The communication difference operator $A-GG^* $ in (\ref{curl_rot_Mat}) is actually the matrix representation of $\mbox{\bf curl}~\mbox{rot}$ except that there are some modifications near the boundaries (because of the boundary treatment). Moreover, the operators $A-GG^*$ and $(A-GG^* )G$ are the actually discrete analogies of (\ref{Lap_iden})-(\ref{divcurl_iden}).

\begin{table}[h]\label{eig_check}
\begin{center}
\begin{tabular}{|c|c|c|c|c|}
\hline
  BC type    & $n_x$  & $n_y$    & DOF     & No. non-unitary eigs       \\
\hline
pure Diri          &16    &16   &736    &60  \\
\hline
pure Diri          &32    &32   &3008   &124     \\
\hline
x-periodic y-Diri  &16    &32   &1520   &31    \\
\hline
x-periodic y-Diri  &32   &64    &6112   &63    \\
\hline
\end{tabular}
\vspace{2mm} \caption{Number of the nonunitary eigenvalues of the Schur complement under different types of boundary conditions.
}
\end{center}
\end{table}
To verify the conclusions in Theorem \ref{th:eig_steady}, we present some numerical experiments. In Table \ref{eig_check}, we list the number of the non-unitary eigenvalues under different types of the boundary conditions and different numbers of the grid points. The computation is based on uniform mesh partitions (with the meshsize equals to 1) on rectangular domains. We form the Schur complement and call Matlab function to calculate the eigenvalues. For purely Dirichlet boundary condition, the total DOF is equal to $(n_x-1)n_y+n_x(n_y-1)+n_x n_y$. If the boundary condition along $x$- direction is periodic, the total DOF is equal to $n_x n_y+n_x(n_y-1)+n_x n_y$. If the boundary conditions are purely periodic, the total DOF is equal to $3 n_x n_y$. The eigenvalue calculations are based on (\ref{preconditioned_sys}) and (\ref{eigen_eqn}). From Table 3.1, we see that the results verify the theoretical predictions. We comment here that there are similar results in \cite{kobelkov2000effective} and the reference therein \cite{kobelkov1994numerical}. The discussions therein are based on the comparisons of the dimensions of functionals spaces under different boundary condition types.

Note that when the Stokes problem is stationary, $P_4$ is similar to the BFBt preconditioner in \cite{elman1999preconditioning}. The analysis of $P_4^{-1} M$ in the steady state Stokes case can be found in \cite{elman1999preconditioning}, we therefore omit it.


\subsection{Analysis of the preconditioners for the unsteady models}

For the time dependent Stokes problem, we can again get started from (\ref{preconditioned_sys}) and (\ref{eigen_eqn}). As $P_2^{-1} M$ and $P_3^{-1} M$ share the same spectral with that of $P_1^{-1} M$. The following spectral analysis is valid for all these three preconditioners. 

To simplify our presentation, we analyze an equivalent system like that in \cite{mardal2004uniform}. Multipling (\ref{Saddle_operator}) by $\frac{\Delta t}{\rho}$, denoting $\frac{\mu \Delta t}{\rho}$ as $\epsilon^2$ and absorbing the scaling $\frac{\Delta t}{\rho}$ into the pressure term, we see that the resulting system is again a saddle point problem of the form (\ref{Saddle_operator}) with $A= I - \epsilon^2 L$ and
$\mathbb{S}^{-1}= (- L^c)^{-1}+\epsilon^2 I$. The corresponding projection method based preconditioner is
\begin{equation}\label{proj_precon_newscale}
P_1^{-1}=\left[\begin{array}{cc}
I & - G\\
0      & I-\epsilon^2 L^c
\end{array}\right]
\left[\begin{array}{cc}
I     &  0 \\
0     &  (-L^c)^{-1}
\end{array}\right]
\left[\begin{array}{cc}
I     &  0 \\
-D&  -I
\end{array}\right]
\left[\begin{array}{cc}
A^{-1}       & 0 \\
0            & I
\end{array}\right],
\end{equation}
and the preconditioned system is (\ref{preconditioned_sys}) with
\begin{eqnarray}\label{app_est}
\mathbb{S}^{-1} S & = &(I - \epsilon^2 L^c)(-L^c)^{-1}(-DA^{-1}G) \cr
&\approx& (I - \epsilon^2 L^c)(-L^c)^{-1}(-D) (G(I - \epsilon^2 L^c)^{-1}) \approx I.
\end{eqnarray}
As mentioned before, when boundary condition is purely periodic, the "=" holds in (\ref{app_est}).

By using the connections between the eigenvalues of $\mathbb{S}^{-1}S$ and a generalized eigenvalue problem, we will reduce the spectral analysis to the analysis of the coefficients of a quadratic polynomial. We have the following theorem for the unsteady Stokes problem.

\begin{theorem}
\label{th:time_analysis}
For the unsteady Stokes problem, the eigenvalues of the preconditioned system have uniform lower and upper bounds, which are independent of the mesh refinement and the parameter $\epsilon$. More precisely, $\beta^2\le \lambda(\mathbb{S}^{-1} S) \le 1$.
\end{theorem}
\begin{proof}
To investigate the eigenvalues of $\mathbb{S}^{-1} S$, we consider the following generalized eigenvalue problem \cite{elman2005finite}.
\begin{equation}\label{time_eig_sys}
\left[\begin{array}{cc}
A & G\\
-D& 0
\end{array}\right]
\left[\begin{array}{c}
{\bf u} \\
p
\end{array}\right]=\sigma
\left[\begin{array}{c}
 A {\bf u} \\
-\mathbb{S} p
\end{array}\right].
\end{equation}
If $\sigma \ne 0$, then $p= - \frac{1}{\sigma} \mathbb{S}^{-1} G^* \ub$. Substituting it into the first equation of (\ref{time_eig_sys}), we have
\begin{equation}\label{quad_eqn}
A \ub - \frac{1}{\sigma} G \mathbb{S}^{-1} G^* \ub = \sigma A \ub.
\end{equation}
For (\ref{quad_eqn}), taking the inner product with $\ub$ and dividing by $(A \ub, \ub)$, we obtain the following quadratic polynomial.
$$
\sigma^2 - \sigma + \frac{(G \mathbb{S}^{-1} G^* \ub, \ub)}{(A \ub, \ub)} =0.
$$
We note that $\sigma$ satisfies a quadratic polynomial equation and therefore its values can be found by using Vieta's formulas. To prove $\sigma$ has uniform lower and upper bounds, it is sufficient to prove the generalized Rayleigh quotient
$$
\frac{(G \mathbb{S}^{-1} G^* \ub, \ub)}{(A \ub, \ub)}, \quad  \forall {\bf u} \notin \mbox{Ker}(G^*)
$$
has uniform lower and upper bounds.

Firstly, there holds
\begin{equation}\label{upper_bound}
\frac{(G \mathbb{S}^{-1} G^* \ub, \ub)}{(A \ub, \ub)} \le 1, \quad \forall {\bf u} \notin \mbox{Ker}(G^*).
\end{equation}
To prove (\ref{upper_bound}), we note that
$$
A - G \mathbb{S}^{-1} G^* = I - \epsilon^2 L  - G (-L^c)^{-1} G^* - \epsilon^2 G G^*
= (I - G (-L^c)^{-1} G^*)+ \epsilon^2 (-L-G G^*).
$$
Using Theorem \ref{th:eig_steady} and (\ref{Schur_bound}), $-L-G G^*$ is symmetric positive semi-definite. Moveover, because ${\bf u} \notin \mbox{Ker}(G^*)$, there exists a $p \in {\bf P}$ such that ${\bf u}=Gp$. Therefore,
$$
(I - \epsilon^2 L  - G (-L^c)^{-1} G^* \ub, \ub)=(G p, G p) - ((G (G^* G)^{-1} G^*) G p, G p) =0.  \quad  \forall {\bf u} \notin \mbox{Ker}(G^*).
$$
It follows that
\begin{equation}\label{ker_iden}
\ds
\frac{((I - G (G^* G)^{-1} G^*) \ub, \ub)}{(\ub, \ub)}   = 0,   \quad \forall {\bf u} \notin \mbox{Ker}(G^*).
\end{equation}
Therefore, (\ref{upper_bound}) holds.

To estimate the lower bound, we check
\begin{eqnarray}\label{inf_bound}
\displaystyle
\inf_{{\bf u} \notin \mbox{Ker}(G^*)} \frac{(G \mathbb{S}^{-1} G^* \ub, \ub)}{(A \ub, \ub)} &= & \inf_{{\bf u} \notin \mbox{Ker}(G^*)} \frac{(G (G^* G)^{-1} G^* \ub, \ub)+\epsilon^2 (G G^* \ub, \ub)}{(\ub, \ub)+\epsilon^2(-L \ub, \ub)} \cr
& = & \inf_{\ub \notin \mbox{Ker}(G^*)} \frac{\frac{(G (G^* G)^{-1} G^* \ub, \ub)}{(G G^* \ub, \ub)}+\epsilon^2}{\frac{(\ub, \ub)}{(G G^* \ub, \ub)}+\epsilon^2 \frac{(-L \ub, \ub)}{(G G^* \ub, \ub)}} \cr
& \ge & \inf_{\ub \notin \mbox{Ker}(G^*)} \frac{\frac{(G (G^* G)^{-1} G^* \ub, \ub)}{(G G^* \ub, \ub)}+\epsilon^2}{\frac{(\ub, \ub)}{(G G^* \ub, \ub)}+\frac{1}{\beta^2} \epsilon^2} \cr
& = & \inf_{\ub \notin \mbox{Ker}(G^*)} \frac{\beta^2 \frac{((G^* G)^{-1} G^* \ub, G^* \ub)}{(\ub, \ub)}+\beta^2 \epsilon^2 \frac{(G^* \ub, G^* \ub)}{(\ub, \ub)}} {\beta^2+ \epsilon^2 \frac{(G^* \ub, G^* \ub)}{(\ub, \ub)}} \cr
& = & \inf_{\ub \notin \mbox{Ker}(G^*)} \frac{\beta^2 +\beta^2 \epsilon^2 \frac{(G^* \ub, G^* \ub)}{(\ub, \ub)}} {\beta^2+ \epsilon^2 \frac{(G^* \ub, G^* \ub)}{(\ub, \ub)}}.
\end{eqnarray}
In the above proof, we have used (\ref{proj_precon_newscale}) and (\ref{ker_iden}). Note that $\forall {\bf u} \notin \mbox{Ker}(G^*)$, $\frac{(G^* \ub, G^* \ub)}{(\ub, \ub)} > 0$, the right hand side of (\ref{inf_bound}) is a decreasing function of $\epsilon^2$. Let $\epsilon^2 \rightarrow +\infty$, we conclude that
$$
\displaystyle
\inf_{{\bf u} \notin \mbox{Ker}(G^*)} \frac{(G \mathbb{S}^{-1} G^* \ub, \ub)}{(A \ub, \ub)}  \ge \beta^2.
$$
It follows that $\sigma$ has uniform lower and upper bounds. Note that there holds $\lambda(\mathbb{S}^{-1} S) = \sigma (1- \sigma)$ \cite{elman2005finite}, the eigenvalues of $\mathbb{S}^{-1} S$ have uniform lower and upper bounds. To be more precise, $\beta^2\le \lambda(\mathbb{S}^{-1} S) \le 1$.  \qquad\end{proof}

We comment here that the analysis in this work is much more simple and straightforward than those in \cite{mardal2004uniform, mardal2011preconditioning, olshanskii2006uniform, olshanskii2012multigrid}, which are based on abstract theory and for FEM discretization. To further illustrate the differences of the preconditioners using the MAC discretization and the FEM discretization, we consider the above preconditioners in the case when the Stokes model degenerates to the inviscid model, i.e., $\mu=0$, or equivalently, $\Delta t \rightarrow 0$. By using the redefined scale,
$$
M = \left[\begin{array}{cc}
I   & G\\
-D  & 0
\end{array}\right] \quad \mbox{and} \quad
P_1^{-1}=\left[\begin{array}{cc}
I & - G\\
0      & I-L^c
\end{array}\right]
\left[\begin{array}{cc}
I     &  0 \\
0     &  (-L^c)^{-1}
\end{array}\right]
\left[\begin{array}{cc}
I     &  0 \\
-D&  -I
\end{array}\right]
\left[\begin{array}{cc}
I^{-1}       & 0 \\
0            & I
\end{array}\right].
$$
We note that $P_4$ degenerates to $P_1$. Furthermore,
\begin{eqnarray}\label{preconditioned_ell}
P^{-1}_1 M &=& \left[\begin{array}{cc}
I        & (I-G(L^c)^{-1}D)I^{-1} G\\
0        & (L^c)^{-1} D I^{-1} G
\end{array}\right]
=
\left[\begin{array}{cc}
I        & (I-G(L^c)^{-1}D)I^{-1} G\\
0        & I
\end{array}\right] =
\left[\begin{array}{cc}
I        & 0 \\
0        & I
\end{array}\right].
\end{eqnarray}
Similarly,
\begin{eqnarray}\label{preconditioned_ell_P2}
P^{-1}_2 M &=&
\left[\begin{array}{cc}
I     &  0 \\
0     &  (-L^c)^{-1}
\end{array}\right]
\left[\begin{array}{cc}
I     &  0 \\
-D&  -I
\end{array}\right]
\left[\begin{array}{cc}
I^{-1}  & 0 \\
0       & I
\end{array}\right]
\left[\begin{array}{cc}
I   & G\\
-D  & 0
\end{array}\right]
=
\left[\begin{array}{cc}
I        & G\\
0        & I
\end{array}\right].
\end{eqnarray}
It is easy to verify that $(P^{-1}_2 M - I)^2=0$. We note that the calculations of (\ref{preconditioned_ell}) and (\ref{preconditioned_ell_P2}) do not depend on the types of the boundary conditions. Furthermore, as  GMRES method possesses the Galerkin property \cite{eiermann2001geometric}, when the system degenerates to the mixed form of an elliptic operator, it  converges in 1 iteration using $P_1$ and converges in 2 iterations using $P_2$. Interesting thing is that the conclusion holds true no matter what type of boundary conditions are imposed. One of the reasons is that the MAC discretization of identity operator is exactly the identity matrix. In comparison, the FEM discretization will lead to velocity mass matrix or pressure mass matrix, which is not always an identity matrix.

\section{Numerical Experiments}

In this section, we conduct numerical experiments to further compare the performance of the different preconditioners. The computational domain is $[0, L_x] \times [0, L_y]$ with $L_x=L_y=L=2^{6}$. We fix $\mu =1$, $h_x=h_y=h=1$ and $\Delta t= 0.5 \frac{h^2}{\mu }= 0.5$. The stopping criterion of the GMRES method is set to be
$$
\frac{||{\bf r}_n||}{||{\bf r}_0||} < 1.0 \times 10^{-10}.
$$
Here ${\bf r}_n$ is the residual at the $n$-th step and ${\bf r}_0$ is the initial residual. Cholesky factorization is applied to find $A^{-1}$ and $(- L^c)^{-1}$. The following is the expression of the Taylor vortex flow.
\begin{equation}\label{TaylorVortex}
 \left \{
  \begin{array} {lll}
u({\bf x}, t)&=1 - 2e^{-8\pi^2\mu t/L^2}\cos{\left(\frac{2\pi(x-t)}{L}\right)}\sin{\left(\frac{2\pi(y-t)}{L}\right)},  \\
v({\bf x}, t)&=1 + 2e^{-8\pi^2\mu t/L^2}\sin{\left(\frac{2\pi(x-t)}{L}\right)}\cos{\left(\frac{2\pi(y-t)}{L}\right)}, \\
p({\bf x}, t)&= -e^{-16\pi^2\mu t/L^2}\left[\cos{\left(\frac{4\pi(x-t)}{L}\right)} + \cos{\left(\frac{4\pi(y-t)}{L}\right)}\right].
  \end{array}
 \right.
\end{equation}
As (\ref{TaylorVortex}) is the solution of the unforced Navier-Stokes equation, we move $({\bf u}^k \cdot \nabla){\bf u}^k$ to the right hand side of (\ref{Time_Stokes}). Then, we obtain a forced Stokes problem. Different types of boundary conditions are imposed to check their effects on the performance of the preconditioners. The boundary values are provided according to the expression (\ref{TaylorVortex}). However, we let $\rho$ vary. (We comment here that (\ref{TaylorVortex}) is the exact solution only when $\rho=1$).

\begin{table}[h]
\begin{center}
\begin{tabular}{|c||ccc|ccc|ccc|}
\hline
  BC type         &       & pure Diri     &      &     & x-periodic y-Diri     &     &    & pure periodic  & \\
\hline
$\rho$ &$N_1$ &$N_3$ &$N_4$  &$N_1$ &$N_3$ &$N_4$ &$N_1$ & $N_3$ &$N_4$  \\
\hline
\hline
$100$      &4  &5     &2    &4  &4    &3    &1   &2   &1 \\
\hline
$10$       &6   &6    &3    &5  &6    &3    &1   &2   &1\\
\hline
$1$        &9   &9    &6    &7  &8    &5    &1   &2   &1\\
\hline
$0.1$      &11 &11   &10    &9  &9   &10    &1   &2   &1\\
\hline
$0.01$     &12 &13   &17    &9  &10  &13    &1   &2  &1\\
\hline
$0$        &17 &19   &24    &11 &13  &19    &1   &2  &1\\
\hline
\end{tabular}
\vspace{2mm} \caption{The number of iterations of the GMRES method using different preconditioners for the unsteady Stokes problems with constant density and constant viscosity.}
\end{center}
\end{table}

In Table 4.1, we summarize the numerical results. Note that one can compare the performance using the left preconditioned system $P_2^{-1} M$ with that using the right preconditioned system $MP_3^{-1}$, we will report one of them. We use $N_1$, $N_3$ and $N_4$ to denote the number of iterations using $P_1$, $P_3$ and $P_4$ respectively. From the experiments, we observe that $P_1$ provides slightly better results than that of $P_3$, in particular when the boundary conditions are non-periodic. The reason is that the two preconditioned systems have the same spectral but the $(1, 2)$ block of $P_1^{-1} M$ contains a projection operator. $P_4$ is slightly better when the viscous term is weak (i.e., when the system is more close to the mixed form of an elliptic operator). However, in other regime, the advantage of $P_4$ is not obvious and it needs more operation cost for implementing $P_4^{-1}$. We omit the results for the case $\rho \rightarrow \infty$, because the theoretical analysis has predicted that GMRES method, using $P_1$ converges in 1 iteration, using $P_2$ or $P_3$ converges in 2 iterations. When the solvers for $A^{-1}$ and $(-L^c)^{-1}$ become inexact and the coefficients of the Stokes problems are variable, we refer the readers to our recent work \cite{cai2013efficient}.

\section{Concluding Remarks}

In this work, some projection method based precondtioners for models of incompressible flow are presented. The uniform bounds of the eigenvalues of the preconditioned systems are derived for both the steady and unsteady cases. For the preconditioners using the the approximate Schur complement, we build up the connections between the projection method based preconditioners with the Cahouet-Chabard type preconditioners. The analysis in this work demonstrates the effects of the boundary treatment.
Compared with the corresponding block triangular preconditioners, we see that the projection method based preconditioners provides slightly better approximation properties because the off-diagonal block of the preconditioned system contains a projection operator. For the preconditioners using the exact Schur complement, the study is similar to those Schur complement based preconditioners \cite{ipsen2001note, murphy2000note}. When inexpensive approximations of $A^{-1}$ and of the Schur complement $\mathbb{S}^{-1}$ or $S^{-1}$ exist, these projection method based preconditioners are of practical use. We comment here that the projection method based preconditioner can be realized using other stable discretizations \cite{arnold2006differential, hyman1997adjoint, turek1999efficient}, for instance, finite element discretization or spectral method. Furthermore, the projection method based precoditioners can be extended to other saddle point problems including the mixed form of the elastic problems, optimal control problems and the mixed problems in electromagnetics \cite{benzi2005numerical, brezzi1991mixed}. For a more general saddle point problem of the form
\begin{equation}\label{general_saddle}
M=
\left[\begin{array}{cc}
A    & G\\
-D   & C
\end{array}\right],
\end{equation}
one can also devise the projection method based preconditioners. We note that the Schur complement for (\ref{general_saddle}) is
\begin{equation}\label{new_Schur}
S=C - (-D) A^{-1} G.
\end{equation}
The projection method based preconditioner can be devised by combing the derivation in this paper and the discussion in \cite{ipsen2001note}. If the exact Schur complement is used, then the corresponding projection method based preconditioner again takes the form (\ref{P1_ext}) with $S$ given in (\ref{new_Schur}). If the approximate Schur complement is used, then one obtains $P_1^{-1}$ of the form (\ref{better_P1}). We comment here that for differential operator based saddle point problems, $\mathbb{S}^{-1}$ can be derived by using Fourier analysis approach. 

{\bf Acknowledgments}. The author would like to express his sincere thanks to Aleksandar Donev, Boyce E. Griffith and Olof Widlund for their stimulating discussions.

%

\penalty-8000


\end{document}